\documentclass[a4 paper,10pt,twoside]{article}
\usepackage[centertags]{amsmath}
\usepackage{amssymb}
\usepackage{amsfonts}
\usepackage{amsthm}
\usepackage{newlfont}
\usepackage{amscd}
\usepackage{tikz, graphicx}
\usepackage[top=1.7in, bottom=1.7in, left=1.5in, 
right=1.5in]{geometry}
\title{\Large Corrigendum to: Zero divisor graph of a lattice with respect to an ideal}
\author{}
\date{}
\theoremstyle{plain}
\newtheorem{theorem}{Theorem}[section]
\newtheorem{proposition}[theorem]{Proposition}
\newtheorem{lemma}[theorem]{Lemma}

\newtheorem{example}[theorem]{Example}
\theoremstyle{definition}
\newtheorem{definition}[theorem]{Definition}
\begin{document}
\maketitle
\begin{center}
\textbf{Ahmed Gaber$^{1}$ and Mona Tarek $^{2}$}\vspace{0.15 cm}\\
\small{$^{1,2}$Department of Mathematics\\Faculty of Science, Ain 
Shams University, Egypt\\
$^{1}$a.gaber@sci.asu.edu.eg\vspace{0.15cm}\\ 
$^{2}$Mona.Saad@sci.asu.edu.eg\vspace{0.15cm}\\ }
\end{center}
\vspace{0.1 cm}
\begin{abstract} 
In this paper, we point out several errors in [M.Afkhami, K.Khashyarmanesh and K.Nafar, {\it  Zero divisor graph of a lattice with respect to an ideal}, Beitr Algebra Geom {\bf 56} (2015), 217-225.]. In the previous article, Afkhami claimed that the intersection of all prime ideals belonging to an ideal $I$ of a distributive lattice $L$ equals to $I$. In this corrigendum, a counterexample for this sentence is constructed. Afkhami demonstrate a counterexample to a certain theorem which satisfies the necessary and the sufficient condition of the theorem. On the other hand, we reform many proofs in Afkhami's article.       
\end{abstract}
\begin{flushleft}
\end{flushleft}
\begin{flushleft}
\end{flushleft}
\section{\textbf{I}ntroduction}

\hspace*{0.4 cm}
Among the results that Afkhami et al. show in $[1]$, we state the following. For a lattice $L$, $ \Gamma (L)$  is the graph whose vertex set is the set \begin{center}
         $ \{x \in L ; \ x \wedge y =0 $ for some non-zero element $ y \in L \}$.
    \end{center}
    Two distinct vertices are adjacent if and only if $ x\wedge y = 0 $.
    \begin{definition}$[1,$ Definition $3.1]$
     Let $I$ be an ideal of a lattice $L$. Introduce the zero divisor
     graph of $L$ with respect to $I$ notated by $\Gamma_I (L)$, as follows 
     \begin{center}
         $ \{x \in L\setminus I ; \ x \wedge y =0 $ for some 
         non-zero element $ y \in L\setminus I \}$.
    \end{center}
    Two distinct vertices are adjacent if and only if $ x\wedge y \in I$. 
    \end{definition}
    Afkhami et al. claim that if $I=\{0\}$. Then $\Gamma_I (L)$ is isomorphic to $\Gamma (L)$. Consider the following example;
    \begin{example}
 Consider the following Hasse diagram  $L$  in Figure (1). Then $(L, ^\vee, ^\wedge ) $ is a bounded distributive lattice. Obviously, $V(\Gamma (L))=\{0, a, b, c, x \}$. As $ 0\wedge z=a\wedge b= a\wedge x= c\wedge b= c\wedge x= 0 $. Also, $V(\Gamma_I (L))=\{a, b, c, x \}$. Then $|V(\Gamma (L))|=5 \not =|V(\Gamma_I (L))|=4$. Hence $\Gamma_I (L)$ is not isomorphic to $\Gamma (L)$.  
 \begin{center}
\begin{tikzpicture}[scale=.5]
  \node (max) at (0,4) {$1$};
  \node (z) at (-2,2) {$z$};
  \node (d) at (2,2) {$d$};
\node (a) at (-4,0) {$a$}; \node (y) at (0,0) {$y$};
  \node (b) at (4,0) {$b$};
  \node (c) at (-2,-2) {$c$};
\node (x) at (2,-2) {$x$};
  \node (min) at (0,-4) {$0$};
\draw (min) -- (c) -- (a) -- (z) -- (max) -- (d) -- (b) --
  (x) -- (min);
\draw (c) -- (y) -- (x); \draw (z) -- (y) -- (d);
\end{tikzpicture}
\end{center}
 \begin{center}
     Fig. (1): $L$ 
 \end{center}
 In fact $\Gamma_I (L)$ is  isomorphic to non zero zero gragh $\Gamma (L)$.
 \end{example}
     Afkhami presents the following results;
  \begin{proposition}$[1,$ Proposition $3.2]$
   If $I$ is a proper filter of a lattice $L$, then $ \Gamma_I (L)$ 
   is connected with $diam( \Gamma_I (L)) \leq 3 $. Moreover 
   $gr(\Gamma_I (L)) \leq 7$, providing that $ \Gamma_I (L)$ contains a cycle.
    \end{proposition}
    For an ideal $I$ of a lattice $L$ and let $x \in L$, we set 
    \begin{center}
        $(I:x) = \{z \in L : z \wedge x \in I \}$.
    \end{center}
    It is easy to see that if $L$ is distributive, then $(I:x)$ is an ideal of $L$.
    \begin{lemma}$[1,$ Lemma $3.3]$
   For a proper ideal $I$ of a distributive lattice $L$. If 
   $a-x-y $ is a path in $\Gamma_I(L)$, then either $I \cup \{x \}$ is an ideal of $L$, or $a-x-y$ is contained in a cycle of length 
   at most $4$.
    \end{lemma} 
    In $[1,$ Theorem $3.4]$, Afkhami constructs the proof by cases. In fact case $(4)$ is impossible to happen. In this corrigendum theorem 3.4 is presented with a reformed proof. 
\begin{theorem} $[1,$ Theorem $3.4]$
    For a distributive lattice $L$ and an ideal $I$ of $L$. If
    $\Gamma_I(L)$ contains a cycle , then the core $K$ of $\Gamma_I(L)$ is a union of $3$-cycles or $4$-cycles. Moreover, if $|V(\Gamma_I(L))| \geq 3 $, then any vertex in $\Gamma_I(L)$ is either a vertex of the $K$ or a vertex of degree one.
\end{theorem}    
 \begin{proof}
 Let $a$ be an arbitrary element in $K$ such that $a$ does not belong  to any $3$-cycles or $4$-cycles in $\Gamma_I(L)$. Suppose that $a$ is in a cycle $ a-b-c-d-...-x-a$ of length greater than $4$. By Lemma 1.4,  $I \cup \{a \}$ is a ideal. Obviously $a \wedge d \in  I \cup \{a \} $ and $a \wedge d \not \in  I $ as $a$ and $d$ are not adjacent, then $a \wedge d = a$. Similarly, $a \wedge c = a$. Thus $a \wedge (d \wedge  c) = a  \in I$. Which is a contradiction. Moreover, assume that $|V(\Gamma_I(L))| \geq 3 $. Let $ x $ be an  element in $V( \Gamma_I(L)) $ such that $x$ does not belong to $K$ neither a vertex of degree one. Suppose that $x$ is of degree $n$ for natural number $n$. Hence $x$ is adjacent to $n$ distinct vertices $ a,b,e,f,...$. Since $\Gamma_I(L)$ contains a cycle and by proposition 1.3,
$a-x-b-c-d-b$ is a path in $\Gamma_I(L)$. By Lemma 1.4, $I \cup \{x \}$ is an ideal of $L$ and by an argument similar to the one used in the first paragraph we get to $ (c \wedge d) \wedge x = x \in I $. Which is a contradiction. 
 \end{proof}
  Afkhami mentions in the proof of the previous theorem that the 
  case $(4)$: $a-x-y-b$ is a path in $\Gamma_I(L)$, where $a$ is a 
  vertex of degree one and $b \in K$ is reduced by the case $(3)$: 
  $a-x-b$ is a path in $\Gamma_I(L)$, where $a$ is a vertex of 
  degree one and $b \in K$. Which is not 
   true. In fact, case $(4)$ is impossible. Assume case $(4)$, by 
   the first part of theorem 1.5, $b$ at least in a cycle of length 
   $3$. Then $a-x-y-b-d-e-b $, hence the length between $a$ and $d$ is 
   more than 3. Which contradicts proposition 1.3.\\
   
   For every $a \in L$, the notation ${[a]}^u$ stands for the set $\{ x\in L ; a \leq x \}$.
   \begin{proposition}$[1,$ Proposition $3.6]$
     Let $L$ be a distributive lattice and $I$ be an ideal of $L$. If
     $ \bigcap_{i \in I}{[a]}^u =\{ 1\} $, then $\Gamma_I(L)$ has no cut point.
    \end{proposition}
  Afkhami constructs an example $[1,$ Example $3.7]$ to show that the distributivity of 
  $L$ is a necessity in the previous proposition. Unfortunately, the example does not satisfy 
  the condition $ \bigcap_{i \in I}{[a]}^u =\{ 1\} $.
  \begin{example}$[1,$ Example $3.7]$
    Consider $I =\{ \phi , \{4\},\{4,5\}, \{4,5,6\},... \}$. Set 
    $L=I \cup \{\phi , \Bbb{N}, \{1\}, \{1,2\},\{3\}\}$. Then $L$ is a lattice under inclusion with $0= \phi$ and $1=\Bbb{N}$ but $L$ is not distributive. However $ \bigcap_{i \in I}{[a]}^u = \{\Bbb{N} - \{ 1, 2, 3\},\Bbb{N} \}$.
  \end{example}

\section{Corrigendum to $\sqrt{I} = I$}
Given any ideal $I$ of a lattice $L$, the notation $\sqrt{I}$ stands for the intersection of all prime ideals belonging $I$. Afkhami mistakenly define that a prime ideal $P$ of $L$ is called a prime ideal belonging to an ideal $I$ if $I \subseteq P$. In fact, a prime ideal $P$ of $L$ is called a prime ideal belonging to an ideal $I$ if $P \subseteq I$.  Afkhami also claims the following; 
 \begin{proposition}$[1,$ Proposition $3.12]$
  If $L$ is a distributive lattice, then $\sqrt{I}= I$. 
 \end{proposition}
 the previous proposition is not true. Consider the following example
 \begin{example}
 Consider the Hasse diagram in example 1.2 $L$. Take $ I=(Z]= \{0,c,a,x,y,z\} $, let $ P =\{0, c, a \} $. Clearly, $P$ and $I$ are the only prime ideals contained in $I$, hence $\sqrt{I}= I \cap p = P$.
 \end{example}
 Kindly note that this proposition is used mainly in $[1,$ Corollary $3.14]$. Moreover,  Afkhami presents a set of equivalent conditions in $[1$, Theorem $3.13]$. In which this propositon is used  specifically in the proof of $(ii)\implies(iii)$. In the following,  theorem 3.13 is presented with a corrigendum to the proof of $(ii)\implies(iii)$.
 \begin{theorem}$[1,$ Theorem $3.13]$  Let $L$ be a distributive lattice. Then the following statements are equivalent.
 \item{(i)} $\chi(\Gamma_I(L)) \leq \infty$.
 \item{(ii)} $\omega(\Gamma_I(L)) \leq \infty$.
 \item{(iii)}The ideal $I$ is the intersection of a finite number of prime ideals.
 \end{theorem}
 \begin{proof}
 $(ii) \implies (iii)$ This follows directly from $[2,$ Theorem $3.13]$ by using the duality theorem. 
 \end{proof}

\end{document}